\documentclass[onefignum,onetabnum]{siamart190516}

\usepackage{amsmath}
\usepackage{amsfonts}
\usepackage{amssymb}
\usepackage[colorinlistoftodos,prependcaption,textsize=footnotesize]{todonotes}
\usepackage{lmodern}
\usepackage[T1]{fontenc}
\usepackage{graphicx}
\usepackage{epstopdf}
\usepackage{algorithmic}
\usepackage[T1]{fontenc}
\usepackage{tikz,caption}
\usepackage[caption=false,font=footnotesize]{subfig}
\usepackage{pgfplots}
\pgfplotsset{compat=newest}
\pgfplotsset{plot coordinates/math parser=false}
\usepackage{url}
\usepackage{color}
\usepackage{booktabs}
\usetikzlibrary{plotmarks}
\usetikzlibrary{external}
\usepgfplotslibrary{external} 

\newlength\figureheight
\newlength\figurewidth 

\ifpdf
\DeclareGraphicsExtensions{.eps,.pdf,.png,.jpg}
\else
\DeclareGraphicsExtensions{.eps}
\fi

\usepackage{xcolor}

\newcommand{\R}{\mathbb{R}}

\def\abs#1{\left|#1\right|} 
\def\norm#1{\left\|#1\right\|}

\newcommand{\Eb}{\mathbf{E}}
\newcommand{\Mb}{\mathbf{M}}
\newcommand{\Wb}{\mathbf{W}}
\newcommand{\Ab}{\mathbf{A}}
\newcommand{\Pb}{\mathbf{P}}

\newcommand{\Ib}{\mathbf{I}}
\newcommand{\Lb}{\mathbf{L}}
\newcommand{\Db}{\mathbf{D}}
\newcommand{\Kb}{\mathbf{K}}

\newcommand{\cc}{{c}_0}
\newcommand{\V}{\mathcal V}

\newcommand{\E}{\mathcal E}
\newcommand{\D}{\mathcal D}
\newcommand{\K}{\mathcal K}

\newcommand{\dx}{\mathrm d x}
\newcommand{\ddx}{\frac{\mathrm d}{\dx}}
\newcommand{\by}{\boldsymbol{y}}
\newcommand{\bp}{\boldsymbol{p}}
\newcommand{\bv}{\boldsymbol{v}}
\renewcommand{\bf}{\boldsymbol{f}}
\newcommand{\bu}{\boldsymbol{u}}

\def\minres{{{\sc Minres}}}

\def\gmres{{{\sc Gmres}}}

\numberwithin{theorem}{section}

\newcommand{\TheTitle}{Optimization of a partial differential equation on a complex network} 
\newcommand{\ShortTitle}{Optimal control of PDEs on networks} 
\newcommand{\TheAuthors}{Martin~Stoll and Max Winkler}

\headers{\ShortTitle}{\TheAuthors}
\title{\TheTitle}

\author{Martin Stoll\thanks{Technische Universit\"at Chemnitz, Department of Mathematics, Chair of Scientific Computing, 09107 Chemnitz, Germany, \email{martin.stoll@mathematik.tu-chemnitz.de}}\and
  Max Winkler\thanks{ Technische Universit\"at Chemnitz, Department of Mathematics, Chair of Numerical Methods for PDEs, 09107 Chemnitz, Germany \email{max.winkler@mathematik.tu-chemnitz.de}}
}
\usepackage{amsopn}

\ifpdf
\hypersetup{
  pdftitle={\TheTitle},
  pdfauthor={\TheAuthors}
}
\fi

\begin{document}

\maketitle

\begin{abstract}
Differential equations on metric graphs can describe many phenomena in the physical world but also the spread of information on social media. To efficiently compute the solution is a hard task in numerical analysis. Solving a design problem, where the optimal setup for a desired state is given, is even more challenging. In this work, we focus on the task of solving an optimization problem subject to a differential equation on a metric graph with the control defined on a small set of Dirichlet nodes. We discuss the discretization by finite elements and provide rigorous error bounds as well as an efficient preconditioning strategy to deal with the large-scale case. We show in various examples that the method performs very robustly.
\end{abstract}

\begin{keywords}
Complex Networks, Optimal Control, Preconditioning, Saddle Point Systems, Error Estimation
\end{keywords}

\begin{AMS}
65F08, 65N30, 49J20, 35R02
\end{AMS}
\section{Introduction}
Graphs and networks are ubiquitous in the modeling of physical phenomena, the representation of data, and other applications \cite{berkolaiko2013introduction}. While the graph Laplacian is a crucial tool in the analysis of such networks \cite{Chu97,romano2017little,spielman2007spectral,VLu07}, there are many examples where it is often not sufficient to only reflect the binary relationships of being connected or not.

In order to address further challenges one can extend the concept of a graph to a so-called metric graph. This is a graph where an edge length is assigned to each edge and we also define a differential operator, often called the \textit{Hamiltonian}, on the graph domain. Metric graphs with PDEs defined on them have become an important modeling tool with the development of numerical schemes currently lacking behind the groundbreaking analytical and modeling work done in (applied) mathematics and other disciplines. A recent paper by Benzi and Arioli \cite{ariolifinite} seems to be one of the first general works devoted to numerical schemes on metric graphs with the discussion from discretization to efficient solution. There are some recent works dealing with PDEs on metric graphs, typically tailored to a particular application, see for example the following results related to gas networks \cite{grundel2016numerical,herty2010new,hild2012real,nordenfelt2007spectral,wybo2015sparse,zlotnik2015optimal}.

In this paper, we build upon the work in \cite{ariolifinite} and consider an optimal control problem on a metric graph with the differential equation as the constraint. Such problems have received an enormous amount of attention over the last years when defined on bounded domains in $\R^d$ (cf. \cite{book::hpuu09,book::IK08,book::FT2010} and the references mentioned therein). For optimal control problems on networks the amount of material is very limited and we refer to \cite{leugering2017nonoverlapping} for a challenging gas network application. In the present article we do not want to focus on a particular application but rather discuss the spatial discretization, its errors, the resulting linear systems and the choice of a suitable preconditioned iterative solver.

In such optimal control problems the control can appear as a distributed or boundary control. We here focus on the more challenging case when the control is placed on network nodes with a Dirichlet condition. For related results on discretization strategies for elliptic optimal Dirichlet control problems in bounded domains we refer to \cite{ApelMateosPfeffererRoesch2018,CasasRaymond2006,MayRannacherVexler2013,OfPhanSteinbach2015,PfeffererWinkler2018preprint,Winkler2018preprint}. It is also the purpose of the present paper to extend the ideas developed in these articles to the situation that the state equation is formulated on a metric graph. To this end, we study the necessary optimality conditions based on an adjoint approach and corresponding finite discretizations.

As a first main result, discretization error estimates for the the numerical approximation of
the Dirichlet control as well as for the corresponding states are derived.

Additionally, we introduce a preconditioning strategy that is build upon a Schur-complement approach, a technique that has been applied very successfully for optimal control problems on bounded domains both in the elliptic and parabolic cases \cite{HS10,MW11,PW10,PSW11,DRW08,S11_TDTP}.

The paper is structured as follows. First, we formulate the model state
equation and summarize some preliminary results in
Section~\ref{sec:basics}. Moreover, we introduce a related optimal control
problem and derive the optimality system. Our discretization strategy is
studied in Section~\ref{sec:discretization}. There, we derive rigorous
discretization error estimates for the finite element approximation of the
state equation and the approximation of a discrete Kirchhoff operator which
appears in the discrete optimality system of the studied Dirichlet control problem.
These results are then used for error estimates of the approximate solutions of the optimal control problem.
An efficient solver for the optimality system and corresponding preconditioners are investigated in Section~\ref{sec:preconditioning}. Finally, in Section~\ref{sec:experiments}, we test the theoretically predicted behavior in several numerical experiments. To be more precise, we confirm that the discretization error estimates are sharp and that the preconditioned method is robust.

\section{Optimal Dirichlet control problems on metric graphs}
\label{sec:basics}
\subsection{Differential equations on metric graphs}
We consider an undirected graph $G=(\V,\E)$ consisting of a vertex set
$\V=\left\lbrace v_i\right\rbrace_{i=1}^{n}$ and the edge set $\E$
\cite{Chu97}. Each edge $e\in \E$ connects a pair of nodes $(v^e_a,v^e_b)$ with
$v^e_a,v^e_b\in \V$.  Furthermore, we define a weight function
$w:\V\times \V\rightarrow\mathbb{R}$ satisfying
$w(v,v')=w(v',v)\text{ for all } v,v'\in \V$ for an undirected graph. We further 
assume that this function is positive if there is an edge $e\in\E$ with endpoints $v$ and $v'$ and zero otherwise. The degree of the vertex $v\in \V$ is defined as
$
d(v)=\sum_{v'\in \V}w(v,v').
$
The diagonal degree matrix $\Db\in\R^{n\times n}$ is defined as $\Db_{v,v}=d(v)$.
When using a node $v\in \V$ as an index for a matrix or vector, we actually mean the
corresponding index within $\V$. This notation is more elegant for our purposes.
We can then obtain the graph Laplacian as $\Lb=\Db-\Wb$ with the entries of the weight
matrix $\Wb_{v,v'}$ given by $w(v,v')$ for $v,v'\in \V$.
The Laplacian in this form is rarely used as typically its normalized form
\cite{VLu07} is employed for segmentation purposes.
The normalized Laplacian is defined by
$
	\Lb_s=\Db^{-1/2}\Lb\Db^{-1/2}=\Ib-\Db^{-1/2}\Wb\Db^{-1/2},
$
which is a symmetric matrix that is often used in machine learning and imaging applications \cite{VLu07}. We will rely on the incidence matrix $\Eb\in\R^{n\times m}$ where $m=|\E|$ is the number
of edges and $n=|\V|$ the number of vertices. The graph Laplacian can then
be represented by
\begin{equation*}
	\Lb=\Eb\Eb^{\top}\in\R^{n\times n}.
\end{equation*}
The graph Laplacian is in certain applications not sufficient to describe the intricate relationships as it only reflects information about the nodes being connected and the edge information are stored in a weight. A more sophisticated representation follows from the now introduced concept of metric graphs.

We want to extend the concept of the combinatorial graph and identify each
edge with an interval of length 
$L_e:=\abs{v_a^e-v_b^e}$ on the real line. Such a graph is 
called a \emph{metric graph}. A metric graph can be equipped with a
differential operator, such as e.\,g.\ the one-dimensional Schr\"odinger operator
\begin{equation}
\label{eq:strong_pde}
(\mathcal{H}y)(x):=\left(-\frac{\mathrm d^2}{\dx^2}+\cc(x)\right)y(x)
\end{equation}
with a potential function $\cc$, and on each edge $e\in \E$ one can formulate
a differential equation of the form
\begin{equation*}
	(\mathcal{H}y|_e)(x_e) = f|_e(x_e)\ \text{for all}\ x_e\in (0,L_e),
\end{equation*}
where the functions $f|_e\colon (0,L_e)\to\R$, $e\in \E$, are given source terms.
Here, $x_e$ are local coordinates associated to the edge $e$.
With a slight abuse of notation we will sometimes evaluate $y|_e$ in
one of the vertices $v\in \V$ of $e$. Depending on the orientation of the local coordinate $x_e$
we then mean either $y|_e(0)$ or $y|_e(L_e)$.

Additionally, we can impose boundary or vertex conditions to couple these equations.
There are of course several different vertex conditions and we will distinguish among two different types. First, in vertices $v\in \V_\K\subset \V$, we have \emph{homogeneous
Neumann--Kirchhoff conditions}, i.\,e., there holds
\begin{equation}
\label{eq:Kirchhoff_Neumann_condition}
(\K y)(v):=\sum_{e\in \E_v}\frac{\mathrm d}{\dx}y|_e(v)=0
\end{equation}
with $\E_v$ the edge set incident to the vertex $v$.
In vertices $v\in \V_\D:=\V\setminus \V_\K$ the solution $y$ fulfills \emph{Dirichlet conditions}
\begin{equation}
\label{eq:Dirichlet_conditions}
y|_e(v)=u_v,\ \text{for all}\ e\in \E_v,
\end{equation}
where $u\in \R^{n_\D}$, $n_\D := |\V_\D|$, is a vector containing the Dirichlet data.
Note that if all nodes are Dirichlet nodes then the problem trivially
decouples into a set of one-dimensional problems. Other boundary conditions
are not discussed here but will be subject of future research. For more information on the vertex conditions on metric graphs we refer to \cite{berkolaiko2013introduction,grady2010discrete,shuman2012emerging}.

The Kirchhoff-Neumann conditions are the natural boundary conditions for the
differential operator \eqref{eq:strong_pde}, as for each $v\in \V$ and test functions
$\phi\in \otimes_{e\in\E} C^\infty(e)$ that are continuous in $v$ and vanish in $v'\in \V\setminus\{v\}$, the formula
\begin{align}
\label{eq:partial_integration}
(\K y)(v)\,\phi(v) &= \sum_{e\in \E_v} y|_e'(v)\,\phi|_e(v)\nonumber\\
&= \sum_{e\in \E_v} \int_e \left[y'(x) \,\phi'(x) + y''(x)\,\phi(x)\right]\dx
\end{align}
is fulfilled.
Here and in the following, we use the notation
$\ddx y|_e(x) = y|_e'(x)$ and omit the subscript $e$ unless the context otherwise requires.
The state equation considered throughout this article finally reads
\begin{equation}
\label{eq:strong_form}
	\left\lbrace
	\begin{aligned}
	-y'' + \cc\,y &= f &&\mbox{on all}\ e\in \E,\\
	(\K y)(v)&=0&&\mbox{for}\ v\in \V_\K,\\
	y(v)&=u_v &&\mbox{for}\ v\in \V_\D,
	\end{aligned}
	\right.
\end{equation}
where $f\in \otimes_{e\in \E} L^2(e)$ and $u\in \R^{n_\D}$ are given data.

For an appropriate treatment of this boundary value
problem we require some further notation.
We introduce the function spaces
\[
L^2(\Gamma)=\bigotimes_{e\in \E} L^2(e)\quad \mbox{and}\quad
H^1(\Gamma)=\bigotimes_{e\in \E} H^1(e)\cap C^0(\Gamma)
\]
equipped with the norms
\begin{align*}
\norm{y}^2_{L^2(\Gamma)}&:=\sum_{e\in E}\norm{y}^2_{L^2(e)}, \\	
\norm{y}^2_{H^1(\Gamma)}&:=\norm{y}_{L^2(\Gamma)}^2 + \sum_{e\in \E}\norm{y'}^2_{L^2(e)}.
\end{align*}
The $L^2(\Gamma)$-inner product is denoted by $(f,g)_{L^2(\Gamma)} := \sum_{e\in \E} \int_e f(x)\,g(x)\,\dx$.
To establish the essential boundary conditions we moreover define
\begin{align*}
H^1_D(\Gamma)&:=\{y\in H^1(\Gamma)\colon y(v) = 0\quad \forall v\in
\V_\D\}.
\end{align*}

With the integration-by-parts formula \eqref{eq:partial_integration} one can
derive the weak formulation of \eqref{eq:strong_pde}--\eqref{eq:Dirichlet_conditions}: \emph{Find $y\in H^1(\Gamma)$ with
$y(v) = u_v$, $v\in \V_\D$, such that}
\begin{equation}\label{eq:weak_form}
\quad a(y,w)= (f,w)_{L^2(\Gamma)}\quad \forall w\in H^1_D(\Gamma),
\end{equation}
where the bilinear form $a:H^1(\Gamma)\times H^1(\Gamma)\to \R$ is defined by
\begin{align*}
a(y,w):=\sum_{e\in \E} \int_e \left(y'(x)\,w'(x) + \cc(x)\,y(x)\,w(x)\right) \dx.
\end{align*}
Throughout this article, the potential function $\cc$ belongs to
$L^\infty(\Gamma):=\otimes_{e\in \E} L^\infty(e)$ and fulfills $\cc\ge 0$
a.\,e.\ in $\Gamma$. Moreover, the source term $f$ belongs to $L^2(\Gamma)$.
With the Lax-Milgram-Lemma we then directly conclude that \eqref{eq:weak_form} possesses a unique solution.

\subsection{The optimal control problem}
As an extension of the above we now consider an optimal control problem of the form
\begin{equation}
\label{obj}
\text{Minimize}\quad\frac{1}{2}\norm{y-\bar{y}}_{L^2(\Gamma)}^2+\frac{\beta}{2}\abs{u}_2^2\quad
\text{over}\quad y\in H^1(\Gamma), u\in \R^{n_\D}
\end{equation}
subject to the constraint
\begin{equation}
\label{constraint}
	\left\lbrace
	\begin{aligned}
		-y'' + \cc\,y&=f &&\text{on all}\ e\in \E,\\
		(\K y)(v)&=0&&v\in \V_\K,\\
		y(v)&=u_v &&v\in \V_\D.
	\end{aligned}
	\right.
\end{equation}
Here, $\abs{\cdot}_2$ is the Euclidean norm in $\R^{n_\D}$.
The state equation is understood in the weak sense \eqref{eq:weak_form}.  We
can decompose the state into a part depending linearly on $u$ and a constant
contribution, i.\,e., $y=y_u+y_f$, with $y_u\in H^1(\Gamma)$, $y_u(v) = u_v$ for $v\in \V_\D$,
and $y_f\in H_D^1(\Gamma)$ satisfying 
\begin{equation*}
a(y_u,w) = 0
\quad\text{and}\quad
a(y_f,w) = (f,w)_{L^2(\Gamma)}
\end{equation*}
for all $w\in H^1_D(\Gamma)$.
This decomposition allows to introduce a linear control--to--state operator
$S\colon \R^{n_\D}\to L^2(\Gamma)$ defined by $u\mapsto S(u) := y_u$, 
and one can eliminate the state equation in the optimal control problem
\eqref{obj}--\eqref{constraint}. Then, we arrive at the reduced optimization problem
\begin{equation}
\label{eq:reduced_objective}
\text{Minimize}\quad j(u) := \frac12 \norm{Su+y_f-\bar
y}_{L^2(\Gamma)}^2 + \frac\alpha2 \abs{u}_2^2
\quad \text{s.t.}\quad u\in \R^{n_\D}.
\end{equation}
As $S$ is linear, the functional $j$ is quadratic and hence Fr\'echet
differentiable and convex. As a consequence, by differentiation using the
chain rule, we can derive the optimality condition which is the variational
problem
\begin{equation}\label{eq:opt_cond_bad}
\left(Su+y_f-\bar y, Sw\right)_{L^2(\Gamma)} + \beta\,u^\top w = 0\qquad \forall
w\in \R^{n_\D}.
\end{equation}
Due to the convexity of $j$ this condition is also
sufficient for $u$ being the unique global solution of
\eqref{eq:reduced_objective}.  In order to derive a more handy form of the
optimality condition we investigate the adjoint operator of $S$.
\begin{lemma}
The adjoint operator $S^*: L^2(\Gamma)\to\R^{n_D}$ possesses the
representation $S^* = - \K\circ P$, where $\K$ is defined as in
\eqref{eq:Kirchhoff_Neumann_condition} and $P:L^2(\Gamma)\to H_D^1(\Gamma)$ is
defined by $y\mapsto P(y):= p$ with $p\in H_D^1(\Gamma)$ fulfilling
\begin{align*}
a(v,p) = (y, v)_{L^2(\Gamma)}\qquad \forall v\in H_D^1(\Gamma).
\end{align*}
\end{lemma}
\begin{proof}
Standard regularity results for elliptic problems imply that $P$ even maps onto $H^2(\Gamma)\cap H_D^1(\Gamma)$. Consequently, $p$ solves $-p'' + \cc\,p=y$ almost everywhere in $\Gamma$.
Together with the integration-by-parts formula \eqref{eq:partial_integration}
we obtain for all $z \in \R^{n_\D}$
\begin{equation*}
	(y, Sz)_{L^2(\Gamma)}
	= (-p'' + \cc\,p, Sz)_{L^2(\Gamma)}
	= a(Sz,p) - (\K p)^\top z.
\end{equation*}
Due to $p\in H_D^1(\Gamma)$ there holds $a(Sz,p) = 0$. This implies $S^* y = -\K p$
and proves the assertion.
\end{proof}
As a consequence, we can rewrite the optimality condition
\eqref{eq:opt_cond_bad} in the form
\begin{equation*}
(Su+y_f - \bar y, Sw)_{L^2(\Gamma)} = (S^*(Su+y_f-\bar y), w)_{\R^{n_\D}} =
-(\K p, w)_{\R^{n_\D}},
\end{equation*}
where the so-called \emph{adjoint state} $p\in H^1_D(\Gamma)$ is the weak
solution of
\begin{equation}
\label{eq:adjoint_equation}
\left\lbrace
\begin{aligned}
\left(-\frac{\mathrm d^2}{\dx^2}+\cc\right) p &= y-\bar y
&&\mbox{on all}\ e\in \E,\\
(\K p)(v) &= 0 && v\in \V_\K,\\
p(v) &= 0 && v\in \V_\D,
\end{aligned}
\right.
\end{equation}
with $y=Su+y_f$.
This allows for a reformulation of the optimality condition
\eqref{eq:opt_cond_bad} by means of
\begin{equation}
\label{eq:opt_cond}
\beta\,u-\K p = 0.
\end{equation}
To summarize the previous investigations we state the following theorem:
\begin{theorem}
The pair $(u,y)\in \R^{n_\D}\times H^1(\Gamma)$ is the unique global solution
of \eqref{obj}--\eqref{constraint} if and only if some adjoint state
$p\in H_D^1(\Gamma)$ exists, such that the system
\begin{equation}
\label{eq:optimality_system}
\left\lbrace
\begin{aligned}
y(v)=u_v\ v\in \V_\D,\qquad a(y,w) &= (f,w)_{L^2(\Gamma)} &&\forall w\in
H_D^1(\Gamma), \\
a(w,p) &= (y-\bar y,w)_{L^2(\Gamma)} && \forall w\in
H_D^1(\Gamma), \\
\beta\,u - \K p &= 0 &&
\end{aligned}
\right.
\end{equation}
is fulfilled.
\end{theorem}

\section{Discretization and error analysis}
\label{sec:discretization}
In \cite{ariolifinite} the authors study a finite element discretization of
the variational problem \eqref{eq:weak_form} and their approach will also form the basis of
our investigations. For the convenience of the reader we briefly repeat this discretization approach.
On each edge $e\in\E$ we introduce an equidistant grid with vertices
$\{v_a^e=v_0^e,v_1^e,\ldots,v_{n_e}^e=v_b^e\}$. This induces a decomposition of the one-dimensional edge $e$ into intervals $I^e_k:=(v_k^e,v_{k+1}^e)$, $k=0,\ldots,n_e-1$.
The global finite element space is defined by
\begin{equation*}
	V_h := \{y_h\in C(\Gamma)\colon y_h|_{I_k^e} \in \mathcal P_1,\
	k=0,\ldots,n_e-1, e\in \E\}.
\end{equation*}
Here, the discretization parameter $h>0$ is the maximal length of the
intervals $I_k^e$, $k=0,\ldots,n_e-1$, $e\in \E$.
By $\psi_j^e,\ j=1,\ldots,n_e-1$, and $\phi_v,\ v\in \V$,
we denote the nodal basis functions of $V_h$ fulfilling
$\psi_j^e(v_j^e)=1$ for $j=1,\ldots,n_e-1$ and $\phi_v(v) = 1$ for $v\in \V$.
Hence, each function $y_h\in V_h$ can be represented by
\begin{equation}
\label{eq:ansatz_discretization}
	y_h(x)=\sum_{e\in \E}\sum_{j=1}^{n_e-1}y_j^e\,\psi_{j}^{e}(x)
	+ \sum_{v\in \V} y_v\,\phi_v(x).
\end{equation}
The Galerkin approximation of \eqref{eq:weak_form} then reads:
\emph{Find $y_h\in V_h$ with $y_h(v) = u_v$ for $v\in \V_\D$ such that}
\begin{equation}
\label{Galerkin_form}
	a(y_h,w_h) = (f,w_h)_{L^2(\Gamma)} \qquad
	\forall w_h\in V_{h,D},
\end{equation}
with $V_{h,D}:=V_h\cap H_D^1(\Gamma)$.
Analogous to the continuous case, we split the discrete state $y_h$ into
$y_{u,h}:= S_h u$, where 
\begin{equation*}
	S_h:\R^{n_\D}\to V_{h},
\end{equation*}
denotes the \emph{discrete harmonic extension} (harmonic w.\,r.\,t.\ the Hamiltonian $\mathcal{H}$) of the Dirichlet data $u\in \R^{n_\D}$,
and a function $y_{f,h}\in V_{h,D}$ fulfilling homogeneous Dirichlet conditions in
the nodes $\V_\D$. To be more precise, we have
\begin{align}
	y_{u,h}(v) = u_v,\ v\in \V_\D,\quad a(y_{u,h},w_h) &= 0  &&\forall w_h\in V_{h,D},
	\label{eq:def_harmonic_extension}\\
	a(y_{f,h},w_h) &= (f,w_h)_{L^2(\Gamma)}  &&\forall w_h\in V_{h,D},
	\label{eq:def_yhf}
\end{align}
and a simple computation shows $y_h = y_{h,u} + y_{h,f}$.
Finally, we can formulate the discretized optimal control problem
\begin{equation}\label{eq:discrete_objective}
\text{Minimize}\ J_h(y_h,u_h) := \frac12 \norm{y_h - \bar{y}}_{L^2(\Gamma)}^2
+ \frac\alpha2\abs{u_h}_2^2\ \mbox{over}\ u_h\in \R^{n_\D}, y_h\in V_h,
\end{equation}
subject to
\begin{equation}
\label{eq:discrete_state_equation}
	y_h(v) = u_v,\ v\in \V_\D,\quad
	a(y_h,w_h) = (f,v_h)_{L^2(\Gamma)}\qquad \forall v_h\in V_{h,D}.
\end{equation}
As in the continuous case discussed in Section~\ref{sec:basics} we can derive a
necessary and sufficient optimality condition.
First, we define the solution operator $P_h:L^2(\Gamma)\to V_{h,D}$ of the
discretized adjoint equation by $P_h(y) = p_h$ with
\begin{equation}
\label{eq:def_Ph}
	a(w_h,p_h) = (y,w_h)_{L^2(\Gamma)}\quad\forall w_h\in V_{h,D}.
\end{equation}
The discretized Kirchhoff-Neumann operator
$\K_h\colon V_h\to \R^{n_\D}$ is defined in a variational sense by
\begin{align}\label{eq:discrete_Kirchhoff}
	(\K_h p_h)(v) &= a(\phi_v,p_h) - (y, \phi_v)_{L^2(\Gamma)}
	\qquad
	\forall v\in \V_\D,
\end{align}
with $y\in L^2(\Gamma)$ chosen such that $p_h = P_h(y)$.
\begin{lemma}
\label{lem:adjoint_of_S}
	The adjoint of the operator $S_h:\R^{n_D}\to V_h$ is
	\[S_h^*=-\K_h\circ P_h\colon L^2(\Gamma)\to \R^{n_D}.\]
\end{lemma}
\begin{proof}
Let $y\in L^2(\Gamma)$, $z\in \R^{n_D}$ and $p_h = P_h(y)$ be arbitrary.
We confirm using the definitions of $P_h$, $\K_h$ and $S_h$ as well as the properties
$w_h:=S_h(z) - \sum_{v\in \V_\D} z_v\phi_v\in V_{h,D}$ and $p_h\in V_{h,D}$
\begin{align*}
	(y,S_h(z))_{L^2(\Gamma)} &= (y,S_h(z) - \sum_{v\in \V_\D} z_v\,\phi_v)_{L^2(\Gamma)}
	+ (y,\sum_{v\in \V_\D} z_v\,\phi_v)_{L^2(\Gamma)} \\
	&= a(S_h(z) - \sum_{v\in \V_\D} z_v\,\phi_v, p_h) + (y,\sum_{v\in \V_\D}z_v\,\phi_v)_{L^2(\Gamma)} \\
	&= \sum_{v\in \V_\D} z_v\left(-a(\phi_v, p_h) + (y,\phi_v)_{L^2(\Gamma)}\right) \\
	&= -\sum_{v\in \V_\D} z_v\, \K_h(p_h)(v) = -z^\top \K_h(p_h).
\end{align*}
This implies the desired result as $\K_h(p_h) = (\K_h\circ P_h)y$.
\end{proof}
Analogous to the continuous case investigated in Section~\ref{sec:basics}
we can derive an optimality system for
\eqref{eq:discrete_objective}--\eqref{eq:discrete_state_equation} which reads:
\textit{Find $u_h\in \R^{n_\D}$, $y_h\in V_h$ with $y_h(v)=u_{h,v}$
$\forall v\in \V_\D$, $p_h\in V_{h,D}$, such that}
\begin{subequations}
\label{eq:discrete_opt_system}
\begin{align}
	a(y_h,w_h) &= (f,w_h)_{L^2(\Gamma)}
	&&\forall w_h\in V_{h,D},
	\label{eq:discrete_state_eq}\\
	a(w_h,p_h) &= (y_h-\bar y,w_h)_{L^2(\Gamma)} &&\forall w_h\in V_{h,D},
	\label{eq:discrete_adjoint_eq} \\
	\beta\,u_h - \K_h p_h &= 0. &&
	\label{eq:discrete_opt_cond}
\end{align}
\end{subequations}
This is a discrete version of the optimality system
\eqref{eq:optimality_system}. Note that $\K_h(p_h)\ne \K(p_h)$.
Using the exact Kirchoff-Neumann operator $\K$ in the discrete
optimality system is basically possible, but as we have seen, only the
variational Kirchhoff-Neumann map $\K_h$ yields the favorable property that
the approaches ``first optimize then discretize'' and ``first
discretize then optimize'' lead to the same discrete optimality system.

In the remainder of this section we study discretization error estimates.
Throughout this article, $c$ stands for a generic positive constant which may
have a different value at each occurrence. Moreover, $c$ is always independent
of $h$ and the functions appearing in the estimates.

The most technical part of the error analysis is the proof of an 
estimate for the approximate Kirchhoff-Neumann
operator $\K_h$. The proof is based on a duality argument and requires some properties of the
discrete harmonic extension $S_h$.
\begin{lemma}
\label{lem:boundedness_Bh}
	The discrete harmonic extension $S_h$ fulfills the stability estimate
	\begin{equation}\label{eq:stability_harmonic_extenstion}
		\norm{S_h u}_{H^1(\Gamma)} \le c\abs{u}_2
	\end{equation}
	for arbitrary $u\in \R^{n_\D}$.
\end{lemma}
\begin{proof}
	For technical reasons we define a further
	discrete extension $\tilde S_h\colon \R^{n_\D}\to V_h$ defined by
	\begin{equation*}
		[\tilde S_h u]|_e \in \mathcal P_1\quad \forall e\in \E,
		\quad [\tilde S_h u](v) = \begin{cases}
			u_v &\forall v\in \V_\D,\\
			0 &\forall v\in \V_\K.
		\end{cases}
	\end{equation*}
	Obviously, $\tilde S_h u$ is a first-order polynomial on each edge
	$e\in \E$ and thus, $\tilde S_h u\in V_h$.
	As a consequence, we deduce together with the $V_h$-ellipticity of $a(\cdot,\cdot)$,
	the definition of $S_h$ and the fact that $(S_h-\tilde S_h)\,u\in V_{h,D}$
	\begin{equation}\label{eq:harmon_ext_less_aux_ext}
		\norm{S_h u}_{H^1(\Gamma)}^2 \le c\,a(S_h u, S_h u) = c\,a(S_h u, \tilde S_h u)
		\le c\norm{S_h u}_{H^1(\Gamma)}\norm{\tilde S_h u}_{H^1(\Gamma)}.
	\end{equation}
	For an edge $e\in \E$ with vertices $v$ and $v'$ we easily show the
	following properties.
	If $v,v'\in \V_\K$ there holds $[\tilde S_h u]_e=0$,
	and otherwise the values of the $H^1(e)$-seminorms are
	\begin{align*}
	\norm{(\tilde S_h u)'}_{L^2(e)}&= L_e^{-1/2}\abs{u_v}&&\mbox{if}\ v\in
	\V_\D, v'\in \V_\K,\\
	\norm{(\tilde S_h u)'}_{L^2(e)}&= L_e^{-1/2}\abs{u_v-u_{v'}}&&\mbox{if}\ v,v'\in
	\V_\D.
	\end{align*}
	Analogously, we can compute the $L^2(e)$-parts in the $H^1(\Gamma)$ norm and get
	\begin{align*}
	\norm{\tilde S_h u}_{L^2(e)} &\le c\,L_e^{1/2} \abs{u_v}
	&& \mbox{if}\ v\in \V_\D, v'\in \V_{\K},\\
	\norm{\tilde S_h u}_{L^2(e)} &\le c\,L_e^{1/2} \left(\abs{u_v} + \abs{u_{v'}}\right)
	&& \mbox{if}\ v, v'\in \V_\D.
	\end{align*}	
	After summation over all edges and taking into account
	\eqref{eq:harmon_ext_less_aux_ext} we conclude the property \eqref{eq:stability_harmonic_extenstion}.
\end{proof}
As an important consequence, we also obtain a stability estimate for the
adjoint $S_h^*$ and thus, for the discrete Kirchoff-Neumann operator $\K_h$.
\begin{lemma}\label{lem:stability_Kh}
	For arbitrary $y\in L^2(\Gamma)$ there holds the stability estimate
	\begin{equation*}
		\abs{S_h^*\,y}_2 = \abs{\K_h(P_h y)}_2 \le c\norm{y}_{L^2(\Gamma)}.
	\end{equation*}
\end{lemma}
\begin{proof}
	First, we apply the definitions of $\K_h$ from
	\eqref{eq:discrete_Kirchhoff}
	and $P_h$ from \eqref{eq:def_Ph} taking into account that
	$\sum_{v\in \V_\D} u_v\,\phi_v -	S_h u\in V_{h,D}$, and obtain
	\begin{align*}
		\abs{\K_h(P_h y)}_2
		&= \sup_{\genfrac{}{}{0pt}{}{u\in \R^{n_\D}}{\abs{u}_2=1}}
		\abs{u^\top\,\K_h(P_h y)} \\
		&= \sup_{\genfrac{}{}{0pt}{}{u\in \R^{n_\D}}{\abs{u}_2=1}}\abs{\sum_{v\in \V_\D} u_v \left[a(\phi_v,P_h y)	- (y,\phi_v)_{L^2(\Gamma)}\right]} \\
				&= \sup_{\genfrac{}{}{0pt}{}{u\in \R^{n_\D}}{\abs{u}_2=1}}\abs{ a(S_h u,P_h y)	- (y,S_h u)_{L^2(\Gamma)}}.
	\end{align*}
	The first term on the right-hand side vanishes due to the
	definition of $S_h$, see \eqref{eq:def_harmonic_extension}, and $P_h y\in V_{h,D}$.
	The last term can be bounded by means of the Cauchy-Schwarz inequality
	and the stability estimate for $S_h$ proved in
	Lemma~\ref{lem:boundedness_Bh}.
\end{proof}
Now, we are in the position to derive the following general error estimate:
\begin{lemma}\label{lem:general_estimate}
	Let $(u,y)\in \R^{n_\D}\times H^1(\Gamma)$ and $(u_h,y_h)\in \R^{n_\D}\times V_h$ be the solutions of
	\eqref{obj}--\eqref{constraint} and
	\eqref{eq:discrete_objective}--\eqref{eq:discrete_state_equation}, respectively.
	Then, the error estimate
	\begin{equation*}
		\beta\abs{u-u_h}_{2}
		\le c\left(\norm{y-\tilde y_h}_{L^2(\Gamma)} +
		\abs{\K p - \K_h \tilde p_h}_{2}\right)
	\end{equation*}
	is fulfilled, where $p$ is the adjoint state
	corresponding to $u$, and $\tilde y_h= S_h(u) + y_{f,h}$ and $\tilde p_h= P_h(y-\bar y)$
	are the Ritz projections of $y$ and $p$.
\end{lemma}
\begin{proof}
Subtraction of the optimality conditions \eqref{eq:opt_cond} for the continuous problem
and \eqref{eq:discrete_opt_cond} for the discrete problem and using $y_h = S_h
u_h$ and $p_h = P_h(S_h u_h + y_{f,h}-\bar y)$ leads to
\begin{align}
	\label{eq:general_estimate}
	&\quad\beta(u-u_h) = \K p - \K_h p_h \nonumber\\
	& =(\K p - \K_h(P_h(y-\bar y))
	+ \K_h(P_h(y-(S_h u+y_{f,h}))
	+ \K_h(P_h(S_h(u-u_h)))\nonumber \\
	&= (\K p - \K_h \tilde p_h) + S_h^*(y-\tilde y_h) - S_h^*(S_h(u-u_h)).
\end{align}
Note that in the first step, we simply introduced the intermediate functions $\K_h(P_h y)$ and $\K_h(P_h (S_h u))$.
In the last step we inserted the Ritz projections $\tilde y_h = S_h u+y_{f,h}$
and $\tilde p_h = P_h(y-\bar y)$ of $y$ and $p$, respectively, as well as $S_h^*=-\K_h\circ P_h$ .

Next, we multiply the equation \eqref{eq:general_estimate} by $u-u_h$ and get for the first two terms
on the right-hand side
\begin{align*}
	(\K p - \K_h \tilde p_h)^\top (u-u_h)
	&\le \abs{\K p - \K_h \tilde p_h}_{2} \abs{u-u_h}_{2}, \\
	S_h^*(y-\tilde y_h)^\top (u-u_h)
	&\le c\norm{y-\tilde y_h}_{L^2(\Gamma)}\abs{u-u_h}_2.	
\end{align*}
In the latter estimate we used the stability properties of $S_h^*$ proved
in Lemma~\ref{lem:stability_Kh}.
The last term on the right-hand side of \eqref{eq:general_estimate}, tested
with $u-u_h$, is non-positive and can be neglected.
\end{proof}

It remains to derive error estimates for the two terms on the right-hand side of the general
estimate derived in Lemma \ref{lem:general_estimate}. The first term is a simple $L^2(\Gamma)$-error and an estimate can be simply concluded from the $H^1(\Gamma)$-error estimate
\begin{equation}\label{eq:H1_error_estimate}
	\abs{y-\tilde y_h}_{H^1(\Gamma)} \le c\,h\abs{y}_{H^2(\Gamma)}
\end{equation}
proved in \cite[Theorem 3.2]{ariolifinite} and an application of the
Aubin-Nitsche method. These arguments imply
\begin{equation}\label{eq:L2_error_estimate}
	\norm{y-\tilde y_h}_{L^2(\Gamma)} \le c\,h^2\abs{y}_{H^2(\Gamma)}.
\end{equation}
The second term, namely the error estimate for the discrete Kirchoff-Neumann operator,
is more challenging.
\begin{theorem}
\label{thm:kirchhoff_estimate}
	Let $p\in H_D^1(\Gamma)$ be the solution of
	\begin{equation*}
		a(w,p) = (y,w)_{L^2(\Gamma)}\qquad\forall w\in H_D^1,
	\end{equation*}
	and denote by $\tilde p_h\in V_{h,D}$ its Ritz-projection, i.\,e., $a(w_h,p-\tilde p_h) = 0$ for all $w_h\in V_{h,D}$.
	Then, the error estimate
	\begin{equation}\label{eq:error_estimate_Kirchhoff}
		\abs{\K p - \K_h \tilde p_h}_2 \le c\,h\abs{p}_{H^2(\Gamma)}
	\end{equation}
	is fulfilled provided that $p\in H^2(\Gamma)$.
\end{theorem}
\begin{proof}
	Let $u\in \R^{n_\D}$ be an arbitrary vector.
	Applying the integration-by-parts formula
	\eqref{eq:partial_integration}, which implies
	\begin{equation*}
		(\K p)^\top u = \sum_{v\in \V_\D} u_v \left[a(\phi_v,p) - (y,\phi_v)_{L^2(\Gamma)}\right],
	\end{equation*}
	the definition of $\K_h$ from \eqref{eq:discrete_Kirchhoff} and
	the Galerkin orthogonality using the fact that $\sum_{v\in \V_\D} u_v\,\phi_v - S_h u\in V_{h,D}$, yields
	\begin{align*}
		(\K p - \K_h \tilde p_h)^\top u
		&= \sum_{v\in \V_\D}u_v\,a(\phi_v,p-\tilde p_h)= a(S_h u,p-\tilde p_h)\\
		&\le c\norm{p-\tilde p_h}_{H^1(\Gamma)} \norm{S_h u}_{H^1(\Gamma)}.
	\end{align*}
	Together with the $H^1(\Gamma)$-estimate \eqref{eq:H1_error_estimate}
	and Lemma \ref{lem:boundedness_Bh}
	one arrives at
	\begin{equation*}
		\abs{\K p - \K_h\tilde p_h}_2
		= \sup_{\genfrac{}{}{0pt}{}{u\in \R^{n_\D}}{\abs{u}_2 = 1}}
		\abs{(\K p - \K_h \tilde p_h)^\top u} \le c\,h\abs{p}_{H^2(\Gamma)}.
	\end{equation*}
\end{proof}

Now we are in the position to state the main result of this section.
\begin{theorem}
\label{thm:estimate_optimal_control}
	Let $f, y_d\in L^2(\Gamma)$ and $\beta>0$ be arbitrary. The solutions $(u,y)\in \R^{n_\D}\times H^1(\Gamma)$ and $(u_h,y_h)\in \R^{n_\D}\times V_h$ of \eqref{obj}--\eqref{constraint} and
	\eqref{eq:discrete_objective}--\eqref{eq:discrete_state_equation}, respectively,
	fulfill the error estimate
	\begin{equation*}
		\abs{u-u_h}_2 + \norm{y-y_h}_{H^1(\Gamma)} \le
		c\,h\left(\abs{y}_{H^2(\Gamma)} + \abs{p}_{H^2(\Gamma)}\right)
	\end{equation*}
	with a constant $c>0$ independent of $h$.
\end{theorem}
\begin{proof}
The estimate for the control follows after insertion of the
estimate \eqref{eq:L2_error_estimate} and the result of Theorem
\ref{thm:kirchhoff_estimate} into Lemma \ref{lem:general_estimate}.

To obtain an estimate for the state we apply the triangle inequality
\begin{align*}
	\norm{y-y_h}_{H^1(\Gamma)} \le \norm{y-(S_h u + y_{f,h})}_{H^1(\Gamma)}
	+ \norm{S_h(u-u_h)}_{H^1(\Gamma)},
\end{align*}
where we used the decomposition $y_h = S_h u_h + y_{f,h}$, compare
\eqref{eq:def_harmonic_extension}--\eqref{eq:def_yhf}. Note that $S_h u + y_{f,h}$
is the Ritz projection of $y$ and thus, the first term on the right-hand side can be bounded
by means of \eqref{eq:H1_error_estimate}. An estimate for the second term follows from the stability
of $S_h$ shown in Lemma \ref{lem:boundedness_Bh} and the estimate derived for the controls.
\end{proof}

In the numerical experiments in Section~\ref{sec:experiments} we will confirm that this
convergence rate is sharp. Although the estimate used for the state promises
quadratic convergence, the rate is limited due to the approximation of the
discrete Kirchhoff-Neumann operator. However, there exist ideas based on local
mesh refinement which would allow for quadratic convergence in
\eqref{eq:error_estimate_Kirchhoff} as well. For
further details we refer to \cite{PfeffererWinkler2018preprint} where Dirichlet
control problems for elliptic equations in planar and bounded domains are
studied.
An extension of these ideas to the context of graph networks is subject of
further research.

\section{Structure of the system matrices and preconditioning}
\label{sec:preconditioning}
We now discuss the structure of the discretized equations based on the considerations from the previous section as well as the results given in \cite{ariolifinite}.

\subsection{Assembly of finite element matrices}

We use a numbering of the nodes such that the \textit{interior nodes} come first followed by the original graph nodes, (cf. \eqref{eq:sorting_dofs}). We note that the incidence matrix for the \textit{interior nodes} of an edge $e\in \E$ is structured and can be written as
\begin{equation*}
	\Eb_e=
	\begin{bmatrix}
	-1&1&&&\\
	&-1&1&&\\
	&&\ddots&\ddots&\\
	&&&-1&1\\
	\end{bmatrix}\in\R^{n_e-1\times n_e}.
\end{equation*}

The incidence matrix for all interior nodes on all edges is then given via
\begin{equation*}
	\Eb_{\text i}=\textbf{blkdiag}(\left\lbrace\Eb_e\right\rbrace_{e\in \E})\in\R^{m(n_e-1)\times mn_e},
\end{equation*}
where $n_e$ is the same for all edges for simplicity.
To also incorporate the original graph nodes we now introduce the matrices
\begin{equation*}
	{\Eb}^{+}=\frac{1}{2}(\Eb+\abs{\Eb}),\qquad{\Eb}^{-}=\frac{1}{2}(\Eb-\abs{\Eb}), 
\end{equation*}
with $\abs{\Eb}$ the component-wise absolute value of the incidence matrix $\Eb$ of the original graph. We now want to establish the incidence matrix for the extended graph based on the original graph nodes. For this consider
\begin{equation*}
	\hat{\Eb}_{j}^{+}={\Eb}_{j}^{+}\otimes\underbrace{[1,0,\ldots,0]}_{\R^{1\times n_{e_j}}}=[{\Eb}_{j}^{+},\textbf{0},\ldots,\textbf{0}]\in\R^{n\times n_{e_j}}
\end{equation*}
to incorporate outgoing edges, where  ${\Eb}^{+}_j$ indicates the $j$-th column of the matrix  ${\Eb}^{+}$ and $n_{e_j}=n_e$ is the number of internal nodes on the edge $e_j$. Here, the index $j$ runs from $1$ to $m=|\E|$. For the incoming edges we use
\begin{equation*}
	\hat{\Eb}_{j}^{-}={\Eb}_{j}^{-}\otimes[0,0,\ldots,1]=[\textbf{0},\ldots,\textbf{0},{\Eb}_{j}^{-}]\in\R^{n\times n_{e}}.
\end{equation*}
We see that due to the incorporation of the interior nodes this somewhat stretches the incidence matrix of the original graph. As a result the part of the incidence matrix representing the original nodes becomes
\[
\Eb_{\text v}=
\left[
\hat{\Eb}_{1}^{+}+\hat{\Eb}_{1}^{-},\hat{\Eb}_{2}^{+}+\hat{\Eb}_{2}^{-},\ldots,
\hat{\Eb}_{m}^{+}+\hat{\Eb}_{m}^{-}
\right]\in\R^{n\times mn_e}.
\]
We then obtain the incidence matrix of the extended graph as
\[
\tilde{\Eb}=
\begin{bmatrix}
\Eb_{\text i}\\
\Eb_{\text v}
\end{bmatrix}
\in \R^{m(n_e-1)+n\times mn_e}.
\]
In order to build the finite element matrices we incorporate the mesh-parameter $h_e$ via
\[
\Wb_E=\textbf{blkdiag}\left(\left\lbrace\frac{1}{h_e}\Ib_{n_e}\right\rbrace_{e\in\mathcal{E}} \right)\in\R^{(n_{\text i}+m)\times (n_{\text i}+m)},
\]
where $n_{\text i}=\sum_{e\in \mathcal{E}}(n_e-1)$ and in the case that all
edges have the same number of interior nodes there holds $n_{\text i}=m(n_e-1)$.
As was shown in \cite{ariolifinite} the stiffness matrix $\Ab$ has the following structure
\[
\Ab=
\begin{bmatrix}
\Eb_{\text i}\\
\Eb_{\text v}
\end{bmatrix}
\Wb_E
\begin{bmatrix}
\Eb_{\text i}^{\top}&
\Eb_{\text v}^{\top}
\end{bmatrix}=
\begin{bmatrix}
\Eb_{\text i}\Wb_E\Eb_{\text i}^{\top}&\Eb_{\text i}\Wb_E\Eb_{\text v}^{\top}\\
\Eb_{\text v}\Wb_E\Eb_{\text i}^{\top}&\Eb_{\text v}\Wb_E\Eb_{\text v}^{\top}
\end{bmatrix}
.
\] 
According to \cite{ariolifinite} the mass matrix can be written as
\[
\Mb=\frac{1}{6}\left(\abs{\tilde{\Eb}}\widehat{\Wb}_E\abs{\tilde{\Eb}}^{\top}+\textrm{diag}\left(
\left\lbrace\left(\abs{\tilde{\Eb}}\widehat{\Wb}_E\abs{\tilde{\Eb}}^{\top}\right)_{i,i}\right\rbrace_{i=1}^{n_{\text i}+m} \right)\right)
\]
with
\[
\widehat{\Wb}_E=\textbf{blkdiag}\left(\left\lbrace h_e\Ib_{n_e}\right\rbrace_{e\in\mathcal{ E}}\right)\in\R^{(n_{\text i}+m)\times (n_{\text i}+m)}.
\]
With the discretization of both the mass and the stiffness term we are now
able to obtain the matrix representation of the optimization problem where the
mass and stiffness matrix are split according to free and Dirichlet-control
variables. The mass matrix incorporating the term $\cc$ can be assembled in a similar fashion (cf. \cite{ariolifinite}) and we refer to it as $\Mb_{\cc}$.
Finally, the system matrix representing the discrete bilinear form on the left-hand side of \eqref{Galerkin_form} is denoted by
\begin{equation*}
	\Kb = \Ab + \Mb_{\cc}.
\end{equation*}

\subsection{The discrete optimality system}

We start with the ansatz \eqref{eq:ansatz_discretization},
which allows a representation of a function $y_h\in V_h$ by the coefficient vectors $\by\in \R^N$, $N=m(n_e-1)+n_\K + n_\D$ with $n=n_\K + n_\D$, that are sorted in the form
\begin{equation}\label{eq:sorting_dofs}
\by =
\begin{bmatrix}
\by_I \\ \by_K \\ \by_D
\end{bmatrix}
\hat{=}
\left[
\begin{array}{l}
\text{coefficients related to $\psi_j^e$, $j=1,\ldots,n_e-1$, $e\in \E$} \\
\text{coefficients related to $\phi_v$, $v\in\K$} \\
\text{coefficients related to $\phi_v$, $v\in\D$}
\end{array}
\right].
\end{equation}
In the same way, we may split the matrices $\Kb$ and $\Mb$
into
\begin{equation*}
\Kb = \begin{bmatrix}
\Kb_{II} & \Kb_{IK} & \Kb_{ID} \\
\Kb_{KI} & \Kb_{KK} & 0 \\
\Kb_{DI} & 0 & \Kb_{DD} 
\end{bmatrix}
\quad\mbox{and}
\quad
\Mb = \begin{bmatrix}
\Mb_{II} & \Mb_{IK} & \Mb_{ID} \\
\Mb_{KI} & \Mb_{KK} & 0 \\
\Mb_{DI} & 0 & \Mb_{DD} 
\end{bmatrix},		
\end{equation*}
respectively.
In order to get a description of the discrete optimality system
\eqref{eq:discrete_opt_system} in a matrix-vector notation, we distinguish only among
Dirichlet nodes and free nodes and thus define
\begin{equation*}
\by_F = \begin{bmatrix}
\by_I\\ \by_K 
\end{bmatrix},\quad
\Kb_{FF} = \begin{bmatrix}
\Kb_{II} & \Kb_{IK} \\ \Kb_{KI} & \Kb_{KK}
\end{bmatrix}
,\quad
\Kb_{FD} = \begin{bmatrix}
\Kb_{ID} \\ 0
\end{bmatrix},\quad
\Kb_{DF} = \Kb_{FD}^\top.
\end{equation*}
Analogous definitions are used for the mass matrix $\Mb$.

Furthermore, we define the load vector $\bf$ by $\bf^\top \bv
= (f,v_h)_{L^2(\Gamma)}$ and the vector $\bar\by$ by $\bar \by^\top \bv = (\bar y, v_h)_{L^2(\Gamma)}$.
These vectors are decomposed as well in the form
\begin{equation*}
\bf = \begin{pmatrix} \bf_F \\
\bf_D\end{pmatrix},
\qquad
\bar \by = \begin{pmatrix} \bar \by_F \\
\bar \by_D \end{pmatrix}.
\end{equation*}

With the previously introduced matrices and vectors we can write the optimality system
\eqref{eq:discrete_opt_system} in the form
\begin{equation}\label{eq:discrete_opt_sys}
\begin{bmatrix}
\Mb_{FF} & \Mb_{FD} & \Kb_{FF}^{\top} \\
\Mb_{DF} & \Mb_{DD}+\beta\Ib & \Kb_{FD}^{\top}		\\
\Kb_{FF} & \Kb_{FD} & \mathbf{0} \\
\end{bmatrix}
\begin{bmatrix}
\by_F \\ \bu \\ \bp_F \\
\end{bmatrix}
=
\begin{bmatrix}
\bar \by_F \\
\bar \by_D \\
\bf_F \\
\end{bmatrix}.
\end{equation}
Note that the sign of the adjoint state in this formulation is different than in
\eqref{eq:discrete_opt_system}. This results in a symmetric system matrix.

In fact, the system \eqref{eq:discrete_opt_sys} is a saddle point or KKT matrix \cite{BenGolLie05,book::andy}. For complex networks with many connections it is infeasible to work with direct solvers \cite{book::duff} due to the higher complexity and fill-in issues. While one could employ non-standard conjugate gradient methods \cite{HS10,RS09,sz2007}, we here employ \minres\ \cite{minres} as a tailored scheme for symmetric and indefinite matrices or \gmres\  \cite{gmres} for the case of a nonsymmetric preconditioned systems. Of course, the performance of this scheme will rely on the distribution of the eigenvalues and we need to improve the performance by introducing a suitably chosen preconditioner $\mathbf{P}$. Its design is discussed in the next part.

\subsection{Preconditioning}
\label{sec:preconditioner}

It is well known \cite{preconMGW} that an ideal preconditioner for saddle
point systems is given by
\[
\mathbf{P}_{\mathrm{ideal}}=
\begin{bmatrix}
\Mb_{FF} & \Mb_{FD} & \mathbf{0} \\
\Mb_{DF} & \Mb_{DD}+\beta\Ib & \mathbf{0}		\\
\mathbf{0} & \mathbf{0} & \mathbf{S} \\
\end{bmatrix}
\]
with the Schur-complement defined as
\[
\mathbf{S}=\left[\Kb_{FF} \quad \Kb_{FD}\right]
\begin{bmatrix}
\Mb_{FF} & \Mb_{FD} \\
\Mb_{DF} & \Mb_{DD}+\beta\Ib
\end{bmatrix}^{-1}
\begin{bmatrix}
\Kb_{FF}^{\top}  \\
\Kb_{FD}^{\top} 
\end{bmatrix}.
\]
While this preconditioner is attractive in producing an optimally clustered spectrum of the preconditioned matrix, it is in general very expensive to apply as the matrix $\mathbf{S}$ is typically dense. A more practical choice is given when we consider a block-diagonal approximation
\[
\mathbf{P}_{\mathrm{ideal}}\approx
\begin{bmatrix}
\Mb_{1} &\mathbf{0}& \mathbf{0} \\
\mathbf{0} & \Mb_{2} & \mathbf{0}\\
\mathbf{0} & \mathbf{0} & \mathbf{S}_1 \\
\end{bmatrix},
\]
where 
\[
\begin{bmatrix}
\Mb_{FF} & \Mb_{FD} \\
\Mb_{DF} & \Mb_{DD}+\beta\Ib\\
\end{bmatrix}\approx
\begin{bmatrix}
\Mb_{1} &\mathbf{0} \\
\mathbf{0} & \Mb_{2} 
\end{bmatrix}
\textrm{ and } 
\mathbf{S}_1\approx \mathbf{S}.
\]
To make these approximation more precise, we first focus on approximating the mass matrix block and note that again an ideal preconditioner according to \cite{preconMGW} is 
\[
\begin{bmatrix}
	\Mb_{FF} & \Mb_{FD} \\
	\Mb_{DF} & \Mb_{DD}+\beta\Ib\\
\end{bmatrix}\approx
\begin{bmatrix}
	\Mb_{FF} &\mathbf{0} \\
	\mathbf{0} & \Mb_{DD}+\beta\Ib-\Mb_{DF}\Mb_{FF}^{-1}\Mb_{FD}
\end{bmatrix}.
\]
Now, we use this approximation and the notation
\[
\mathbf{S}_M=\Mb_{DD}+\beta\Ib-\Mb_{DF}\Mb_{FF}^{-1}\Mb_{FD}
\]
to obtain the following approximation of the Schur-complement of the overall KKT system
\begin{align*}
	\mathbf{S}&\approx \left[\Kb_{FF} \quad \Kb_{FD}\right]
	\begin{bmatrix}
	\Mb_{FF} &\mathbf{0} \\
	\mathbf{0} & S_M
	\end{bmatrix}^{-1}
	\begin{bmatrix}
	\Kb_{FF}^{\top}  \\
	\Kb_{FD}^{\top} 
	\end{bmatrix}\\
	&=\Kb_{FF} \Mb_{FF}^{-1}\Kb_{FF}^{\top}+\Kb_{FD}\mathbf{S}_M^{-1}\Kb_{FD}^{\top}.
\end{align*}
Still, the approximation \[\Kb_{FF} \Mb_{FF}^{-1}\Kb_{FF}^\top+\Kb_{FD}\mathbf{S}_M^{-1}\Kb_{DF}\]  is not very practical to work with since we would need to invert a sum of matrix products. We want to derive an efficient approximation to this sum. First,  we look at the following approximations $\Mb_{FF}\approx \mathbf{D}_M:=\textrm{diag}(\Mb_{FF}).$ The usefulness of this approximation for practical purposes is illustrated in Figure \ref{fig:prec_mass_diagonal:eigs}, where we show that using the approximation $\mathbf{D}_M$ causes almost no change in the eigenvalues of the preconditioned mass matrix
\[
\begin{bmatrix}
\Mb_{FF} &\mathbf{0} \\
\mathbf{0} & \mathbf{S}_M
\end{bmatrix}^{-1}\begin{bmatrix}
\Mb_{FF} & \Mb_{FD} \\
\Mb_{DF} & \Mb_{DD}+\beta\Ib
\end{bmatrix}.
\]
We also use $\mathbf{S}_M\approx \mathbf{D}_{\mathbf{S}_M}:=\Mb_{DD}+\beta\Ib-\Mb_{DF}\mathbf{D}_M^{-1}\Mb_{FD}$ and then get
\begin{equation}
\label{eq::preconschur1}
\Kb_{FF} \Mb_{FF}^{-1}\Kb_{FF}^{\top}+\Kb_{FD}\mathbf{S}_M^{-1}\Kb_{FD}^{\top}\approx \Kb_{FF} \mathbf{D}_{M}^{-1}\Kb_{FF}^{\top}+\Kb_{FD}\mathbf{D}_{\mathbf{S}_M}^{-1}\Kb_{FD}^{\top}.
\end{equation} 

\tikzexternaldisable
\begin{figure}
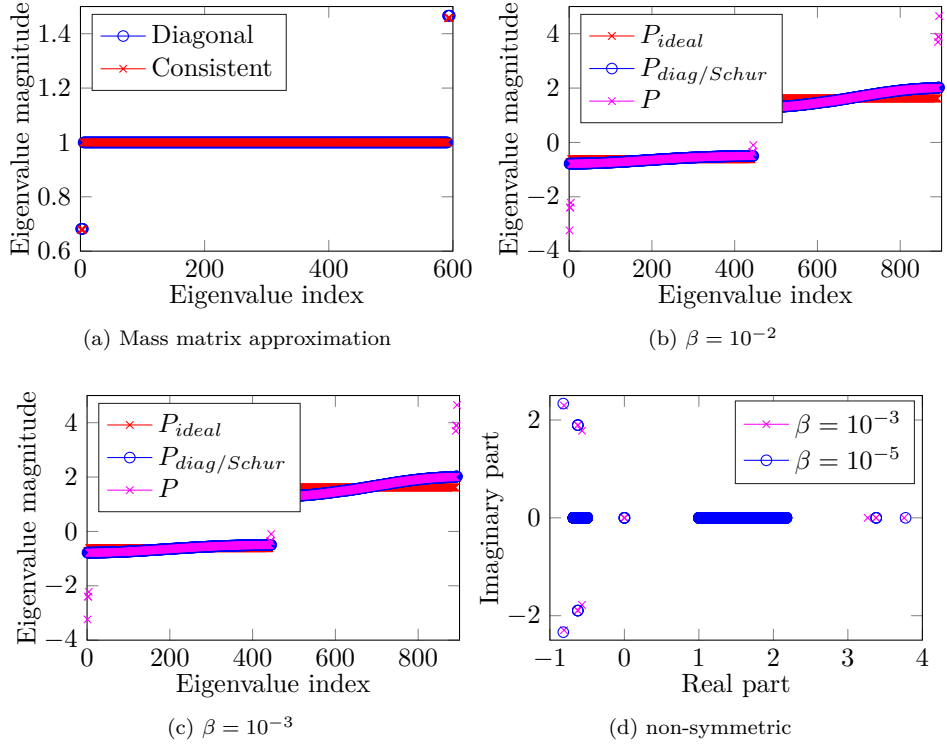

	\begin{center}
		\setlength\figureheight{.25\linewidth} 
		\setlength\figurewidth{.4\linewidth}
		\subfloat[Mass matrix approximation]{
			\input{Mapprox.tikz}\label{fig:prec_mass_diagonal:eigs}}
		\subfloat[$\beta=10^{-2}$]{
			\input{Precon_symm_ex1.tikz}\label{fig:prec_mass_diagonal:Pall1}}
			
		\subfloat[$\beta=10^{-3}$]{
			\input{Precon_symm_ex2.tikz}\label{fig:prec_mass_diagonal:Pall2}}
		\subfloat[non-symmetric]{
			\input{Precon_unsymm_ex1.tikz}\label{fig:prec_mass_diagonal:Pall3}}
			\caption{Eigenvalue plots of the preconditioned matrices. In particular, preconditioned mass matrix (top left), preconditioned saddle point system in symmetric form (top right, bottom left), and preconditioned saddle point system with non-symmetric preconditioner (bottom right)}
		\label{res:prec_mass_diagonal}
	\end{center}
\end{figure}
\tikzexternalenable

The approximation $\Kb_{FF}\mathbf{D}_{M}^{-1}\Kb_{FF}^{\top}+\Kb_{FD}\mathbf{D}_{\mathbf{S}_M}^{-1}\Kb_{FD}^{\top}$ is difficult to work with directly as this involves a sum of terms. We now want to further rework this to get a usable preconditioning scheme
\[
\Kb_{FF} \mathbf{D}_{M}^{-1}\Kb_{FF}^{\top}+\Kb_{FD}\mathbf{D}_{\mathbf{S}_M}^{-1}\Kb_{FD}^{\top}
\approx \left(\Kb_{FF}+\mathbf{N}\right)\mathbf{D}_{M}^{-1}\left(\Kb_{FF}+\mathbf{N}\right)^{\top}.
\]
A matching argument \cite{PW11,PSW11} requires the following
\[
\mathbf{N}\mathbf{D}_{M}^{-1}\mathbf{N}^{\top}=\Kb_{FD}\mathbf{D}_{\mathbf{S}_M}^{-1}\Kb_{FD}^{\top},
\]
which holds true for 
\[
\mathbf{N}=\left(\Kb_{FD}\mathbf{D}_{\mathbf{S}_M}^{-1}\Kb_{FD}^{\top}\right)^{1/2}\mathbf{D}_{M}^{1/2}.
\]
Finally, since we do not want to take the square root of a possibly very large matrix, we obtain the final Schur-complement 
\[
\mathbf{N}=\left(\mathbf{D}_{KDK}\right)^{1/2}\mathbf{D}_{M}^{1/2},
\]
where $\mathbf{D}_{KDK}$ is the lumped version of $\Kb_{FD}\mathbf{D}_{\mathbf{S}_M}^{-1}\Kb_{FD}^{\top}$. While this approach fits well within the preconditioning framework for a symmetric solver like \minres \ \cite{minres} we can alternatively try to avoid the square root and possibly lumping of a matrix. For this lets assume we are willing to employ a non-symmetric solver such as \gmres \  \cite{gmres}. We then again use a matching argument
\[
\Kb_{FF} \Mb_{FF}^{-1}\Kb_{FF}^{\top}+\Kb_{FD}\mathbf{D}_{\mathbf{S}_M}^{-1}\Kb_{FD}^{\top}
\approx \left(\Kb_{FF}+\mathbf{N}_1\right)\Mb_{FF}^{-1}\left(\Kb_{FF}^{\top}+\mathbf{N}_2\right)
\]
with the non-symmetric approach
\[
\mathbf{N}_1\Mb_{FF}^{-1}\mathbf{N}_2=\Kb_{FD}\mathbf{D}_{\mathbf{S}_M}^{-1}\Kb_{FD}^{\top},
\]
which we can achieve by using 
\[
\mathbf{N}_1=\Mb_{FF},\quad \mathbf{N}_2=\Kb_{FD}\mathbf{D}_{\mathbf{S}_M}^{-1}\Kb_{FD}^{\top}.
\]
Note that since we do not need to take square roots here we can work with the matrix $\Mb_{FF}$ and not just its diagonal or lumped version.

Now, we can write down the preconditioner
\[
\Pb=
\begin{bmatrix}
	\mathbf{D}_{M} &\mathbf{0} &\mathbf{0} \\
	\mathbf{0} & S_M&\mathbf{0} \\
	\mathbf{0} &\mathbf{0} & \left(\Kb_{FF}+\mathbf{N}\right)\mathbf{D}_{M}^{-1}\left(\Kb_{FF}+\mathbf{N}\right)^{\top}
\end{bmatrix}
\]
for the symmetric setup and 
\[
\Pb=
\begin{bmatrix}
\mathbf{D}_{M} &\mathbf{0} &\mathbf{0} \\
\mathbf{0} & S_M&\mathbf{0} \\
\mathbf{0} &\mathbf{0} & \left(\Kb_{FF}+\mathbf{N}_1\right)\mathbf{M}_{FF}^{-1}\left(\Kb_{FF}^{\top}+\mathbf{N}_2\right)
\end{bmatrix}
\]
in the unsymmetric case.
For a small example using this in Matlab we see in Figures
\ref{fig:prec_mass_diagonal:Pall1} and \ref{fig:prec_mass_diagonal:Pall2},
where the value of the regularization parameter is varied,  that the
eigenvalues of the preconditioned matrix are nicely clustered for the ideal
preconditioner $\Pb_{\text{ideal}}$, the approximation with Schur-complement
using \eqref{eq::preconschur1} called $\Pb_{\text{diag/Schur}}$, and the matching
based approximation $\Pb$. Note that since we will only work with $\Pb$ the approximation with an expensive Schur-complement $\Pb_{\text{diag/Schur}}$ is never used.
`We also observe that there is almost no change in the eigenvalues even when the parameter $\beta$ changes. Similar robustness is observed when we use the non-symmetric Schur-complement approximation as illustrated in Figure \ref{fig:prec_mass_diagonal:Pall3}.

The efficiency of our preconditioners is discussed in the numerical experiments in the next section.
\section{Numerical experiments}
\label{sec:experiments}
In this section we illustrate how the methodology developed in this paper performs when applied to several challenging datasets. Our implementation is based on MATLAB\textsuperscript{\textregistered}.
For the iterative solvers we rely on the standard implementations of \minres\ and \gmres\ within  MATLAB\textsuperscript{\textregistered}. We run the algorithm until a relative tolerance
of $10^{-8}$ is reached.

\subsection{Preconditioning}
The goal of the preconditioners we designed is to provide fast and robust methods that lead to a method exploiting as much as possible the structure of the saddle point problems presented previously. While for certain smaller and non-complex graphs the use of direct solvers is very efficient, the motivation for constructing preconditioners is the applicability for complex networks and also more challenging differential equations that contain time-derivatives and possibly parameter dependencies. 

The first example we consider is a simple star graph, where we have one internal node
and all remaining nodes are leaf nodes. The solution of a Dirichlet
boundary control problem in such a graph is illustrated in Figure
\ref{fig:star1}. The required iteration numbers of the iterative \gmres-method
using the non-symmetric preconditioner constructed in Section
\ref{sec:preconditioner} are reported in Table~\ref{tab:star_iterations} and
the computational times in Table~\ref{tab:star_time}.
Obviously, our iterative solver is robust,
both with respect to the regularization parameter $\beta$
and the discretization parameter $n_e$.
It can be also observed that our method outperforms an unpreconditioned iterative solver.
As a second example, we consider a finite difference graph of an L-shaped domain,
also illustrated in Figure \ref{fig:L_graph_solution}.
To generate this graph we used the Matlab function ``numgrid('L',N)'', where
$N$ is the number of vertices at the long edges of an L-shaped domain.
In the experiment we used $N=10$ which leads to a graph with $75$ nodes, and
we randomly selected $12$ controllable nodes.
The number of \gmres\ iterations are presented in Table~\ref{tab:fdm}.
As in the previous example, we observe robustness of the iteration numbers with respect to changes in $\beta$ and the number of finite element nodes. We see a slight mesh-dependence for an increasing number of degrees of freedom but the iteration numbers stay rather low. 
\begin{table}[tbh]
	\setlength{\tabcolsep}{4pt}
	\begin{center}
		\subfloat[Star graph -- Iterations\label{tab:star_iterations}]{
		\begin{tabular}{lccccccc} 
		\toprule
		\ $N_{\text{DOF}}$  & 86 & 182 & 374 & 758 & 1526 & 3062 & 6134 \\
		$\beta$  & & & & & & & \\ 
		\midrule 
		$10^{-2}$ & 11 (11) & 25 (35) & 29 (87) & 27 (244) & 36 (727) & 31 (2473) & 31 (5636)\\
		$10^{-3}$ & 16 (16) & 25 (35) & 29 (89) & 34 (243) & 31 (740) & 31 (2474) & 31 (5636)\\
		$10^{-4}$ & 16 (16) & 22 (36) & 24 (90) & 34 (243) & 34 (727) & 31 (2472) & 31 (5548)\\
		$10^{-5}$ & 16 (16) & 22 (36) & 24 (90) & 34 (243) & 33 (729) & 31 (2472) & 31 (5633)\\
		\bottomrule
		\end{tabular}
		}
		
		\subfloat[Star graph -- computational time in seconds\label{tab:star_time}]{
		\begin{tabular}{lcccc} 
		\toprule
		\ $N_{\text{DOF}}$ & 758 & 1526 & 3062 & 6134 \\
		$\beta$  & & &\\ 
		\midrule 
		$10^{-2}$ &  0.01 (0.06) & 0.03 (0.75) & 0.04 (17.00) & 0.09 (353.95) \\
    		$10^{-3}$ &  0.01 (0.04) & 0.02 (0.71) & 0.04 (16.60) & 0.10 (354.36) \\
		$10^{-4}$ &  0.01 (0.05) & 0.02 (0.68) & 0.04 (16.97) & 0.09 (353.03) \\
		$10^{-5}$ &  0.01 (0.04) & 0.02 (0.65) & 0.04 (16.88) & 0.09 (346.52) \\
		\bottomrule
		\end{tabular}
		}
		
		\subfloat[FDM graph -- Iterations\label{tab:fdm}]{
		\begin{tabular}{lccccc} 
		\toprule
		\  $N_{\text{DOF}}$  & 398 & 918 & 1958 & 4038 & 8198 \\
		$\beta$ &  & & & & \\ 
		\midrule 
		$10^{-2}$ & 44 (235) &   90 (563) &  103 (1309) &   94 (3233) &   92 (--) \\
		$10^{-3}$ & 47 (238) &   89 (565) &  106 (1313) &   87 (3241) &   86 (--) \\
		$10^{-4}$ & 46 (246) &   86 (570) &  106 (1309) &   88 (3233) &   84 (--) \\
		$10^{-5}$ & 46 (236) &   92 (566) &  100 (1309) &   94 (3248) &   85 (--) \\
		\bottomrule
		\end{tabular}
		}
\caption{Dependence of the iteration numbers of \gmres\ on the regularization parameter $\beta$
and discretization parameter $N_{\text{DOF}} = n + m\,(n_e-1)$. The values in parentheses are the results
for the unpreconditioned method.}
\end{center}
\end{table}

\begin{figure}[tbh]
	\begin{center}
		\includegraphics[width=.9\textwidth]{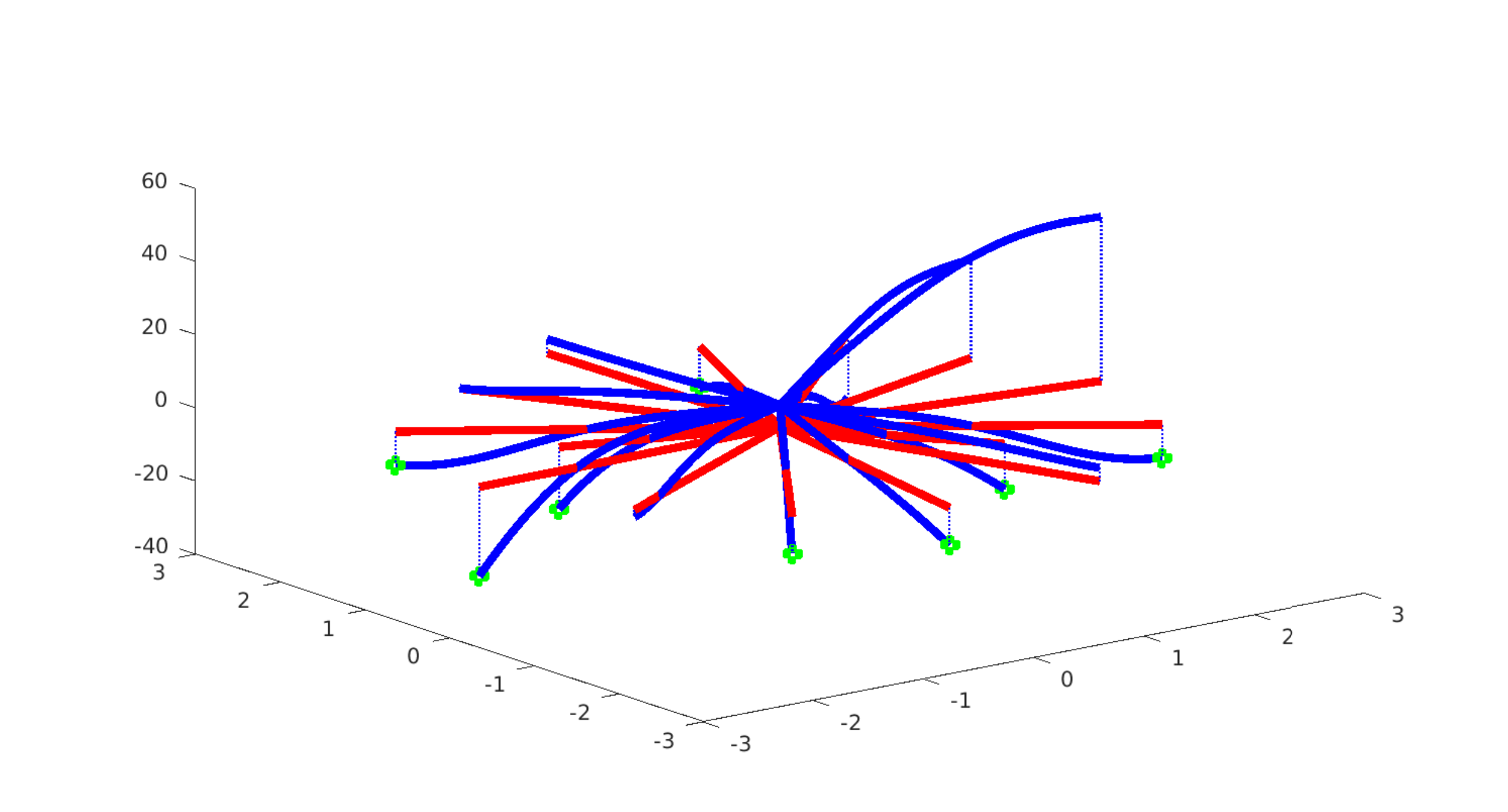}
		\caption{Optimal solution of the Dirichlet control problem in a star-shaped network. The red lines represent the edges of the
graph, the blue lines represent the optimal state and the green stars are the
control nodes.}
		\label{fig:star1}
	\end{center}
\end{figure}
\begin{figure}[tbh]
\begin{center}
\includegraphics[width=.9\textwidth]{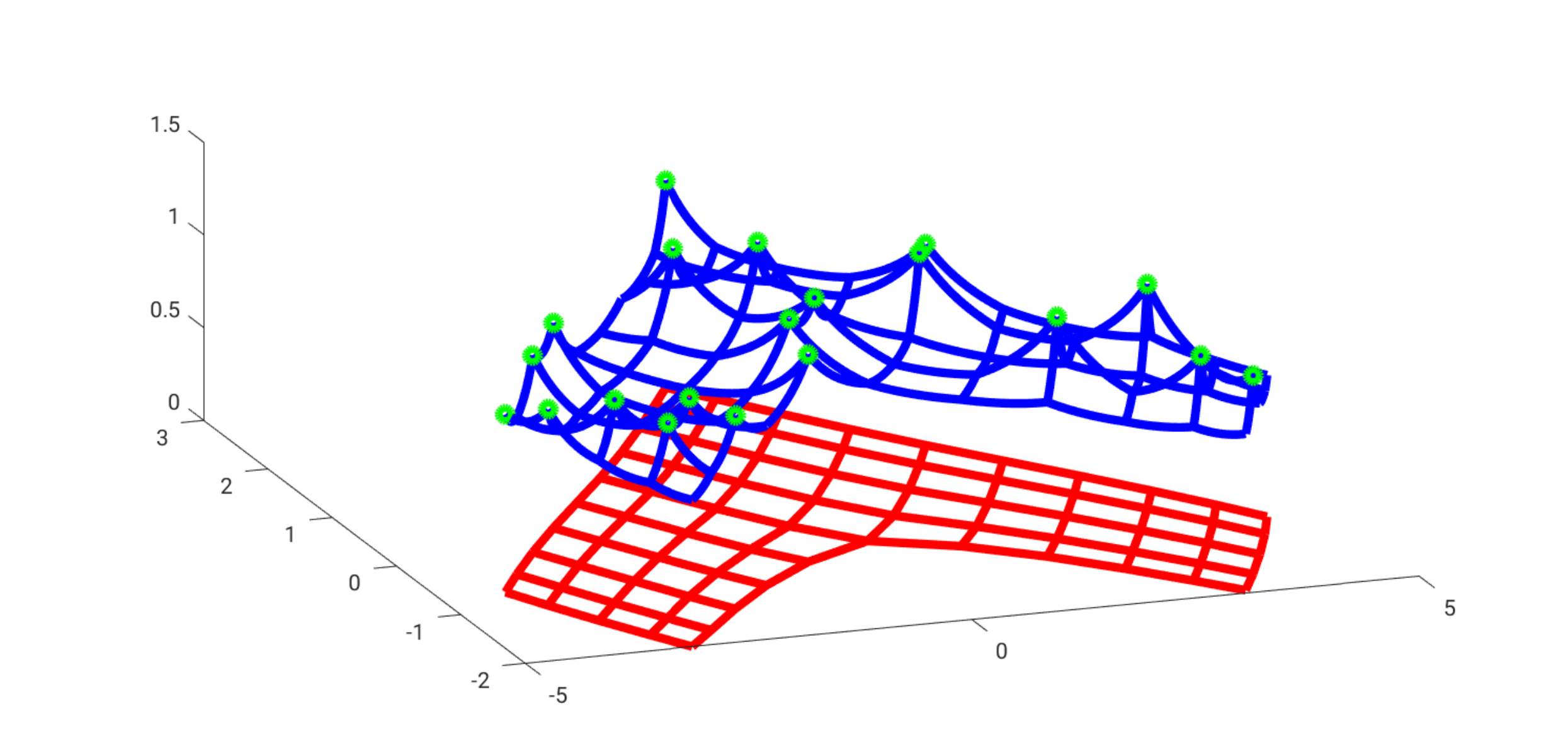}
\end{center}
\caption{Solution of the optimal Dirichlet control problem on an L-shaped finite
difference graph.}
\label{fig:L_graph_solution}
\end{figure}

\subsection{Discretization error estimates}
\label{sec:experiment_discretization}

In this experiment we want to check whether the convergence behavior predicted by Theorem~\ref{thm:estimate_optimal_control} is also observed in experiments. To this end, we use again the FDM graph of an L-shaped domain, as considered already in the previous section.
The input data for the optimal control problem are
$\beta=0.1$, $\bar y\equiv 1$, $f \equiv 1.5$ and $\cc\equiv 2$
and the exact solution is illustrated in Figure \ref{fig:L_graph_solution}.
We repeated the computation of the optimal Dirichlet control problem for
the parameters $n_e=2^k$, $k=1,2,\ldots,13$. The discretization error for the
control and the state in the $L^2(\Gamma)$- and $H^1(\Gamma)$-norm is obtained by comparison of the solution with the solution on the
finest grid with $k=15$. In Table \ref{tab:discretization_error} the absolute
values of the error and the experimental convergence rates are summarized.
Obviously, the discrete controls and states converge linearly towards the exact solution
which confirms that the convergence rates proved in
Theorem~\ref{thm:estimate_optimal_control} are sharp.
\begin{table}[bth]
\begin{center}
\begin{tabular}{lllll}
\toprule
$N_{\text{DOF}}$ & $h$ & $\abs{u-u_h}_2$ & $\norm{y-y_h}_{L^2(\Gamma)}$ &
$\norm{y-y_h}_{H^1(\Gamma)}$ \\ \midrule
8395   & $2^{-6}$ &	4.60e-03	(0.94) & 1.45e-02 (1.00) &	2.74e-02 (1.00) \\
16715  & $2^{-7}$ &	2.34e-03	(0.97) & 7.24e-03 (1.00) &	1.37e-02 (1.00) \\
33355  & $2^{-8}$ &	1.17e-03	(0.99) & 3.61e-03 (1.00) &	6.84e-03 (1.00) \\
66635  & $2^{-9}$ &	5.87e-04	(1.00) & 1.79e-03 (1.01) &	3.40e-03 (1.00) \\
133195 & $2^{-10}$ & 2.89e-04	(1.01) & 8.81e-04 (1.02) &	1.68e-03 (1.01) \\
266315 & $2^{-11}$ & 1.40e-04	(1.04) & 4.26e-04 (1.04) &	8.25e-04 (1.02) \\
532555 & $2^{-12}$ & 6.57e-05	(1.09) & 1.99e-04 (1.09) &	3.95e-04 (1.06) \\
\bottomrule
\end{tabular}
\caption{Absolute error of the discrete controls in $\abs{\cdot}_2$ and states in
$\norm{\cdot}_{L^2(\Gamma)}$ and $\abs{\cdot}_{H^1(\Gamma)}$. The numbers in
parentheses are experimental convergence rates. Here, $N_{\text{DOF}}$ is the number of
degrees of freedom for the state variable, i.\,e., the number of nodes in the
extended graph.}
\label{tab:discretization_error}
\end{center}
\end{table}

\subsection{Other networks}
\label{sec:complex_networks}
We now illustrate on two more examples that the technique presented by us does
apply to more general complex networks. The first example is the road network
of Minnesota\footnote{\url{https://www.cise.ufl.edu/research/sparse/matrices/Gleich/minnesota.html}}
illustrated in Figure~\ref{fig:minnesota}.
Also for this problem, the preconditioner we studied is robust with respect to
$\beta$ and $h$, see Table~\ref{tab:minnesota}.
The second complex network we use is the Facebook Ego Network Dataset\footnote{\url{https://blogs.mathworks.com/loren/2016/02/03/visualizing-facebook-networks-with-matlab/}} shown in Figure~\ref{fig:facebook}.
\begin{table}
	\begin{center}
		\begin{tabular}{lccc} 
			\toprule
			\quad $N_{\text{DOF}}$ &210000  &    421328  &    843980\\
			$\beta$  &  &  & \\
			\midrule 
			$10^{-1}$ &45 &   45 &   61\\
			$10^{-2}$ &39 &   31 &   42\\
			$10^{-3}$ &39 &   29 &   41\\
			$10^{-4}$ &39 &   29 &   41\\
			\bottomrule
		\end{tabular}
		\caption{Iteration numbers of the \gmres-method for the solution of
		an optimal Dirichlet control problem on the Minnesota graph.}
		\label{tab:minnesota}
	\end{center}
\end{table}
\begin{figure}[htb]
	\begin{center}
	\subfloat[Road network of Minnesota\label{fig:minnesota}]{\includegraphics[width=.49\textwidth]{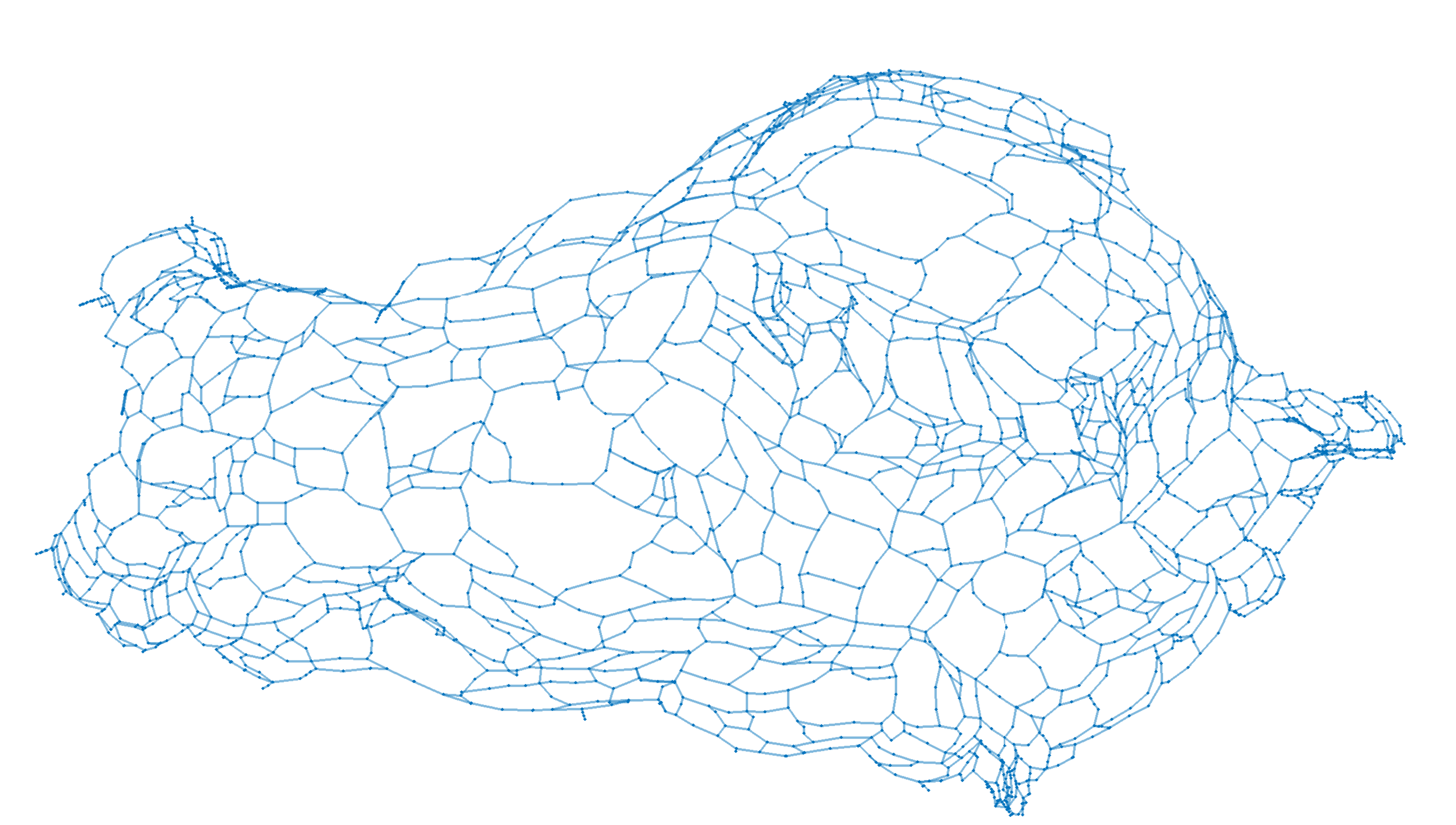}}
	\subfloat[Ego network of Facebook\label{fig:facebook}]{\includegraphics[width=.49\textwidth]{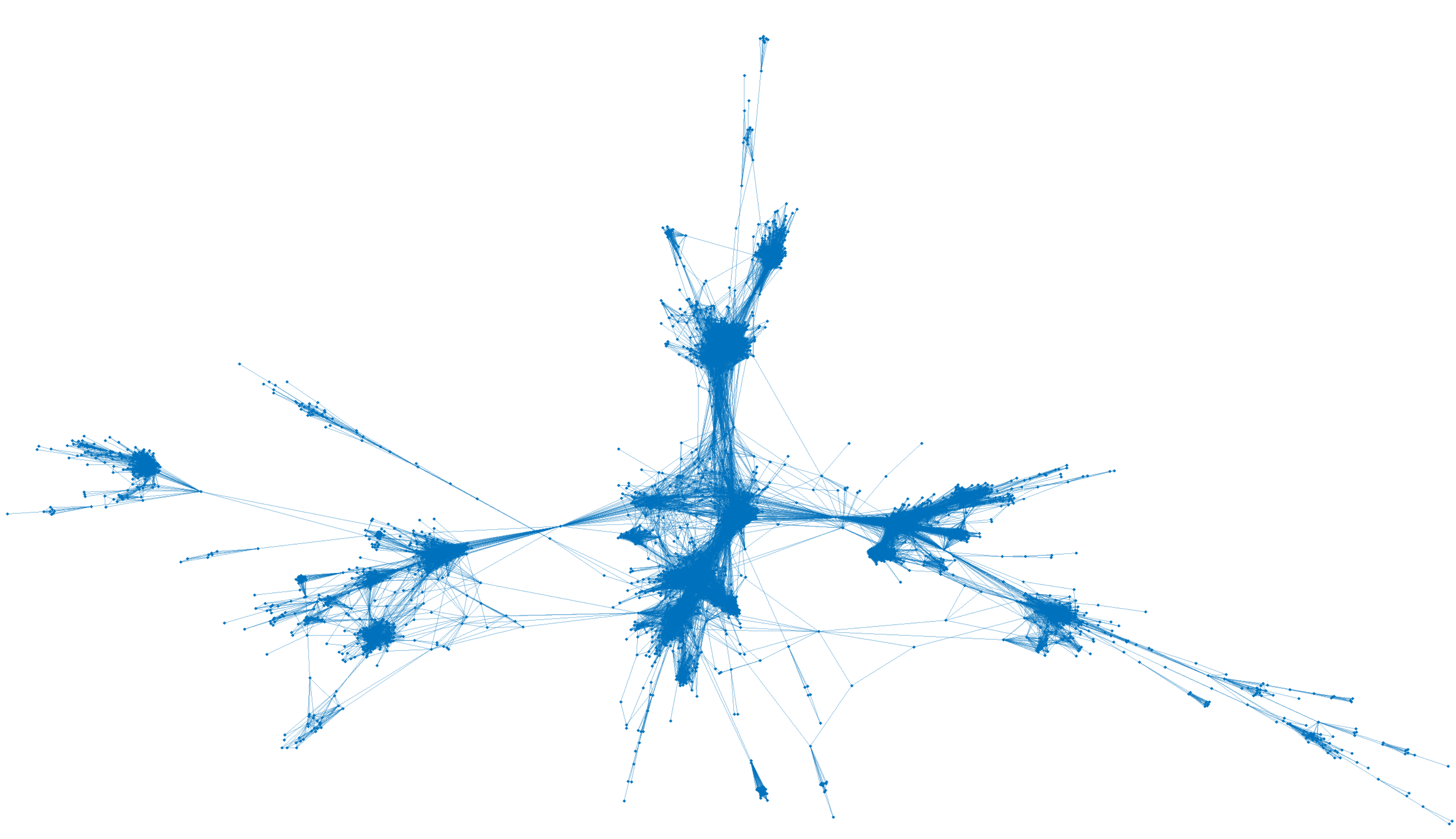}}
	\end{center}
	\caption{Example networks investigated in Section \ref{sec:complex_networks}.}
\end{figure}
We here have $5{,}228{,}870$ degrees of freedom and the solver is able to obtain the solution after $38$ iterations for $\beta=10^{-3}$ and $37$ iterations for $\beta=10^{-4}$.
\section{Conclusion}
In this paper we have discussed the problem of a PDE-con\-strained optimization problem on a complex network. The steady PDE-operator and the objective function were discretized using finite elements. A rigorous error analysis showed first order convergence that we later studied in the numerical experiments section of the paper. We further discussed the matrix structure following \cite{ariolifinite} and proposed a preconditioning strategy for the saddle point system representing the first order conditions. Numerical experiments illustrated that we do indeed observe the proven discretization error and that the developed preconditioned performed robustly with respect to changes of all the relevant parameters,
\bibliographystyle{siam}
\bibliography{../../refs,../../data,../../publications}

\end{document}